\documentclass[a4paper,11pt]{article}
\usepackage[ansinew]{inputenc}
\usepackage[english]{babel}
\usepackage{amssymb}
\usepackage{amsmath}
\usepackage{amsthm}
\usepackage[pdftex]{color, graphicx}
\usepackage{enumerate}

\theoremstyle{definition} \newtheorem{defi}{Definition}[section]

\theoremstyle{plain} \newtheorem{thm}[defi]{Theorem}

\theoremstyle{plain} \newtheorem{lm}[defi]{Lemma}

\theoremstyle{plain} \newtheorem{prop}[defi]{Proposition}

\theoremstyle{plain} \newtheorem{cor}[defi]{Corollary}

\theoremstyle{definition}

\newcommand{\N}[0]{\mathbb{N}}

\newcommand{\Irr}{\textrm{Irr}}

\newcommand{\cd}{\textrm{cd}}

\newcommand{\dl}{\textrm{dl}}
\newcommand{\pos}{\textrm{pos}}
\newcommand{\posmin}{\textrm{pos}_{\min}}
\newcommand{\posmax}{\textrm{pos}_{\max}}
\newcommand{\posind}[2]{[#1:#2]_{\rm{pos}}}

\author{Tobias Kildetoft\footnote{Supported in part by the center of excellence grant 'Center for Quantum Geometry of Moduli Spaces' from the Danish National Research Foundation (DNRF95)} \\ Aarhus University}
\title{Positions of characters in finite groups and the Taketa inequality}
\begin{document}

\maketitle




\section*{Abstract}

We define the position of an irreducible complex character of a finite group as an alternative to the degree. We then use this to define three classes of groups: PR-groups, IPR-groups and weak IPR-groups. We show that IPR-groups and weak IPR-groups are solvable and satisfy the Taketa inequality (ie, that the derived length of the group is at most the number of degrees of irreducible complex characters of the group), and we show that any M-group is a weak IPR-group. We also show that even though PR-groups need not be solvable, they cannot be perfect.

\section{Introduction}

Let $G$ be a finite group and let $\cd(G) = \{\chi(1)\mid\chi\in\Irr(G)\}$ be the set of (irreducible) character degrees of $G$. It is a conjecture that if $G$ is solvable then $\dl(G)\leq|\cd(G)|$ where $\dl(G)$ denotes the derived length of $G$. This inequality is called the Taketa inequality.

One of the first results in the direction of the above inequality was the theorem of Taketa (\cite[Theorem 5.12]{isaacs}) that if $G$ is an M-group then $G$ is solvable and satisfies the Taketa inequality. Some other conditions under which a finite solvable group is known to satisfy the Taketa inequality are $|G|$ odd (\cite[Corollary 16.7]{manzwolf}) and $|\cd(G)|\leq 5$ (\cite[Main Theorem]{lewis01}).

If $G$ is a finite solvable group which does not satisfy any of the above conditions, there are still two known bounds on $\dl(G)$ in terms of $|\cd(G)|$, namely $\dl(G)\leq 2|\cd(G)| - 3$ (\cite[Theorem 1]{kildetoft12}) and $\dl(G)\leq |\cd(G)| + 24\rm{log}_2(|\cd(G)|) + 364$ (\cite[Theorem 3.6]{keller03}). The latter is the better of the two bounds when $|\cd(G)|\geq 588$.

In this paper, we will define the position of an irreducible character of $G$ as an alternative to the degree, and use these positions to define certain classes of groups, which we call position reducible groups (PR-groups), inductively position reducible groups (IPR-groups) and weak IPR-groups.

We then show that any IPR-group is a weak IPR-group and that any weak IPR-group is solvable and satisfies the Taketa inequality. We also show that if $G$ is an M-group then $G$ is a weak IPR-group, and we conjecture that in fact $G$ is an IPR-group.

PR-groups need not be solvable, but we show that if $G$ is a PR-group, then the derived subgroup of $G$ is not perfect (and hence neither is $G$). We also show that if $G$ is a solvable PR-group with at least $6$ character degrees, then $\dl(G)\leq 2|\cd(G)| - 4$.

\subsection*{Acknowledgements}
I would like to thank J\o rn B Olsson for acting as advisor on my master's thesis, in which the ideas of this paper first emerged.

I would also like to thank Mark L. Lewis for reading an early version of the paper and providing helpful comments.

\section{Notation and preliminaries}

In this paper, $G$ is a finite group, and character means complex character. We will use the following notation.
\begin{itemize}
\item $\Irr(G)$ is the set of irreducible characters of $G$.
\item $\rm{Lin}(G)$ is the set of linear characters of $G$ (i.e. the characters of degree $1$).
\item $\Irr(G|G')$ is the set of non-linear irreducible characters of $G$ (following the notation of \cite{isaacsknutson98}).
\item $\cd(G) = \{\chi(1)\mid\chi\in\Irr(G)\}$ is the set of character degrees of $G$.
\item $\varphi^G$ is the character of $G$ induced from $\varphi$ when $\varphi$ is a character of $H$ for some $H\leq G$.
\item $\chi_H$ is the restriction of $\chi$ to $H$ when $\chi\in\Irr(G)$ and $H\leq G$.
\item $[\chi,\psi]$ is the usual normalized inner product of the characters $\chi$ and $\psi$ of $G$. 
\end{itemize}

The following lemmas will be used several times in the paper.

\begin{lm}\label{lm1}Let $H\leq G$ and let $\varphi$ be a character of $H$. 

Then we have $$\rm{ker}(\varphi^G) = \bigcap_{g\in G}\rm{ker}(\varphi)^g$$ In particular, we have $\rm{ker}(\varphi^G)\leq H$.
\end{lm}

\begin{proof}This is Lemma 5.11 in \cite{isaacs}\end{proof}

\begin{lm}\label{lm2}Let $H\leq G$, let $\varphi$ be a character of $H$ and assume that $N$ is a normal subgroup of $G$ with $N\leq \rm{ker}(\varphi)$. 

Then $N\leq \rm{ker}(\varphi^G)$.
\end{lm}

\begin{proof}This follows directly from Lemma \ref{lm1}.\end{proof}

\section{Positions of characters}

In this section, we will define the position of an irreducible character of $G$ as an alternative to the degree. We will then use this to associate certain numbers to arbitrary characters and also to subgroups.

\begin{defi}[Position of a character]Let $\cd(G) = \{f_1,f_2,\dots,f_n\}$ with $f_1
< f_2 < \cdots < f_n$. Let $\chi\in\Irr(G)$ with $\chi(1) = f_i$.
Then the position of $\chi$ is defined to be $\pos(\chi) =
i$.\end{defi}

Note that an alternative (but equivalent) way to define the position of an irreducible character $\chi$ is $\pos(\chi) = |\{i\in \cd(G)\mid i\leq \chi(1)\}|$.

Clearly we have $\pos(\chi)\leq \chi(1)$.

Given some group $G$, if we know $\cd(G)$, then for any $\chi\in\Irr(G)$, we have that $\chi(1)$ and $\pos(\chi)$ provide the same information. But if we look at some character ``in isolation'', then the two numbers give us different information. 

Of course, we have that $\pos(\chi) = 1$ if and only if $\chi(1) = 1$, and if $\chi(1) = 2$ then $\pos(\chi) = 2$. But $\pos(\chi) = 2$ need not imply that $\chi(1) = 2$ (in fact, $\chi(1)$ can be arbitrarily large in this situation).

The concept of position of irreducible characters has in fact been used several places in the literature already, though without giving a name or a notation to it. Examples include the precise formulation of the theorem of Taketa on the solvability of M-groups already mentioned, as well as the normal series $D_i(G)$ (\ref{normalseries}). 

One motivation for looking at positions rather than degrees of the characters is that it allows us to make the following definitions much easier.

\begin{defi}[Taketa-character]Let $\chi\in \Irr(G)$ with $\pos(\chi) = i$. We say that $\chi$ is a Taketa-character if $G^{(i)}\leq \rm{ker}(\chi)$.
\end{defi}

It is clear that any linear character is a Taketa-character.

\begin{defi}[Taketa-group]$G$ is said to be a Taketa-group if all the irreducible characters of $G$ are Taketa-characters.\end{defi}

The mentioned theorem of Taketa can now be stated as ``If $G$ is an M-group then $G$ is a Taketa-group'', and this is the reason for the choice of the name.

Clearly, if $G$ is a Taketa-group then $G$ is solvable and
$G$ satisfies the Taketa inequality (since the intersection of the kernels of all the irreducible characters is trivial).

Unlike degrees, it does not make sense to speak of the position of a character if it is not irreducible. There are, however, two distinguished ``positions'' of any character, which we will be interested in.

\begin{defi}[Maximal and minimal position]Let $\chi$ be a character
of $G$. Let $\psi,\varphi\in\Irr(G)$ be constituents of $\chi$
of largest and smallest degrees, respectively. Then we define
$\posmax(\chi) = \pos(\psi)$ and $\posmin(\chi) = \pos(\varphi)$,
the maximal and minimal position of $\chi$.\end{defi}

We can also use these positions to assign a number to any subgroup of $G$. This number will behave a bit like the index of the subgroup, though with some notable exceptions.

\begin{defi}[Positional index of a subgroup]Let $H\leq G$. We define
$$\posind{G}{H} = \min\limits_{\psi\in\Irr(H)}\{\posmax(\psi^G)\}$$ and call it the
positional index of $H$ in $G$.\end{defi}

Thus, $\posind{G}{-}$ is a function from $\{H\, |\, H\leq G\}$ to
$\{1,2,\dots, |\cd(G)|\}$. It has the following properties.

\begin{prop}\label{properties}Let $K$ and $H$ be subgroups of $G$.
\begin{enumerate}[(1)]
\item If $K\leq H$ then $\posind{G}{K}\geq \posind{G}{H}$.
\item $\posind{G}{G} = \posind{G}{G'} = 1$.
\item $\posind{G}{H} = 1 \Leftrightarrow G'\leq H$.
\item $\posind{G}{\{1\}} = |\cd(G)|$.
\item If $H < G$ then $\posind{G}{H} < [G:H]$.
\end{enumerate}
\end{prop}

\begin{proof}
\begin{enumerate}[(1)]
\item Let $\varphi\in\Irr(K)$ minimize $\posmax((-)^G)$ and let $\psi\in\Irr(H)$ be a constituent of $\varphi^H$. Now $\psi^G$ is a constituent of $\varphi^G$ so $\posmax(\psi^G)\leq \posmax(\varphi^G)$ which proves the claim.
\item By (1) it is enough to show that $\posind{G}{G'} = 1$. But we have that $(1_{G'})^G = \sum\limits_{\chi\in\rm{Lin}(G)}\chi$ so this is clear.
\item One direction is clear by (1) and (2). So assume that $\posind{G}{H} = 1$. This means that there is some $\varphi\in\Irr(H)$ such that $\varphi^G$ is a sum of linear characters. But then $G'\leq\rm{ker}(\varphi^G)\leq H$ by Lemma \ref{lm1} as claimed.
\item This is clear since $(1_{\{1\}})^G$ is the regular character of $G$ which has all irreducible characters as constituents.
\item This follows because $(1_H)^G$ has degree $[G:H]$ and has $1_G$ as a constituent, so the degree of any constituent is strictly less than $[G:H]$ and hence so is the position.
\end{enumerate}
\end{proof}

We will be interested in when it is possible to find a subgroup $H$ such that $\posind{G}{H}$ is small compared to the positions of certain characters of $H$ (made precise below). Note that given some $\varphi\in\Irr(H)$, the number $\posmax(\varphi^G)$ is smallest if $\varphi^G$ splits into as many constituents as possible.

When calculating $\posind{G}{H}$ it would be nice if we knew something about which characters $\varphi\in\Irr(H)$ minimize $\posmax(\varphi^G)$. Unfortunately, it is not easy to say much about this. It is for example not always the case that it is minimized by a linear character. An example of this can be found by looking at the group $G = \rm{SL}_2(3)\times C_2$, which has a subgroup $H$ of order $8$ isomorphic to $Q_8$ such that $\posmax(\varphi^G)$ is minimized by a character $\varphi\in\Irr(H)$ with $\pos(\varphi) = 2$ (and it is strictly larger for all the linear characters of $H$).

%

\begin{defi}[Position reducing tuple (PRT)]Let $H < G$,
$\chi\in\Irr(G)$ and $\varphi\in \Irr(H)$ be a constituent of $\chi_H$. 

$(G,H,\chi,\varphi)$ is said to be a position reducing tuple (PRT) if $\pos(\varphi) + \posind{G}{H}\leq \pos(\chi)$.\end{defi}

We will also call $(G,H,\chi)$ a PRT if there is some constituent $\varphi\in\Irr(H)$ of $\chi_H$ such that $(G,H,\chi,\varphi)$ is a PRT. Note that to check whether $(G,H,\chi)$ is a PRT it is enough to check whether $\posmin(\chi_H) + \posind{G}{H}\leq\pos(\chi)$. The reason we also need to consider PRTs with the constituent $\varphi$ of $\chi_H$ specified is that sometimes, it will be useful to be able to put extra requirements on this constituent.

Also note that allowing $H = G$ in the definition of PRT would not change anything, as $(G,G,\chi)$ can never be a PRT. We have chosen to require $H < G$ for practical reasons that will be apparent when we define IPR-groups (Definition \ref{ipr}).

Part of the motivation behind the definition of a PRT is that if $(G,H,\chi)$ is a PRT, then in some respects, the character $\chi$ behaves as if it was induced from a linear character of $H$ (see also Proposition \ref{monomialcharacter}).

\begin{defi}[Position reducible character (PR-character)]We say that $\chi\in\Irr(G)$ is a position reducible character (PR-character) if there is some $H < G$ such that $(G,H,\chi)$ is a PRT.\end{defi}

\begin{prop}\label{monomialcharacter}If $\chi\in \Irr(G|G')$ is monomial, then $\chi$ is a PR-character.
\end{prop}

\begin{proof}Let $H\leq G$ and $\varphi\in\rm{Lin}(H)$ with $\varphi^G = \chi$. We then see that $(G,H,\chi,\varphi)$ is a PRT, since we have $\pos(\varphi) = 1$ and $\posind{G}{H} \leq \posmax((1_H)^G) < \pos(\chi)$ since $1_G$ is a constituent of $(1_H)^G$ and the latter has the same degree as $\chi$.
\end{proof}

As mentioned, it will sometimes be useful to put a further condition on the constituent $\varphi$ of $\chi_H$ in the definition of a PRT. The precise requirement we will need in this paper is the following.

\begin{defi}[Taketa-PR-character]$\chi\in \Irr(G)$ is said to be a Taketa-PR-character if there is a PRT $(G,H,\chi,\varphi)$ such that $\varphi$ is a Taketa-character.
\end{defi}

Since any linear character is a Taketa-character, we see that any monomial $\chi\in\Irr(G|G')$ is also a Taketa-PR-character (by the proof of Proposition \ref{monomialcharacter}).

The following two results say something about when the derived
subgroup can be used to form a PRT.

\begin{lm}\label{xgr1}Let $\chi\in\Irr(G|G')$ and $\varphi\in\Irr(G')$ be a constituent of $\chi_{G'}$. Then
$(G,G',\chi,\varphi)$ is a PRT if and only if $\pos(\varphi) <
\pos(\chi)$.\end{lm}

\begin{proof}
This follows directly from the fact that $\posind{G}{G'} = 1$.
\end{proof}

The following result has proven surprisingly useful in determining whether specific groups were IPR-groups or weak IPR-groups (see later definitions). In specific cases, this result will often be sufficient to prove that a group is a PR-group, and iterating it by taking further derived subgroups until one gets a nilpotent group will then sometimes allow one to conclude that the group is an IPR-group (using Proposition \ref{supersolvable}).

\begin{prop}\label{xgr2}Let $\chi\in\Irr(G|G')$. If $(G,G',\chi)$ is
not a PRT then there exists a $\varphi\in\Irr(G')$ such that
$\pos(\varphi) \geq \pos(\chi)$ and $\varphi(1)t$ divides $\chi(1)$ where
$t$ is the index of the stabilizer of $\varphi$ in $G$.\end{prop}

\begin{proof}If $(G,G',\chi)$ is not a PRT, then by Lemma \ref{xgr1} we
have that $\posmin(\chi_{G'}) \geq \pos(\chi)$ so there is a
$\psi\in\Irr(G')$ which is a constituent of $\chi_{G'}$ and such
that $\pos(\psi)\geq\pos(\chi)$.

Since $G'$ is normal in $G$, however, we know from Clifford's
Theorem that $\chi_{G'} = e\sum_{i=1}^t\psi_i$ where $e =
[\chi_{G'},\psi]$ and the $\psi_i$ are the different conjugates of
$\psi$ in $G$. In particular, we have that $\chi(1) = \chi_{G'}(1) =
et\psi(1)$, so $\psi(1)t$ divides $\chi(1)$ as claimed.\end{proof}

Another consequence of Lemma \ref{xgr1} is

\begin{prop}\label{taketa2}Let $\chi\in\Irr(G|G')$ be given and assume that
$G''\leq\ker(\chi)$. Then $(G,G',\chi)$ is a PRT.\end{prop}

\begin{proof}By Lemma \ref{xgr1} we just need to show that
$\posmin(\chi_{G'}) < \pos(\chi)$. In fact we claim that
$\posmin(\chi_{G'}) = 1$.

Since $\chi_{G''}$ is a multiple of $1_{G''}$ we get that all
irreducible constituents of $\chi_{G'}$ are constituents of
$(1_{G''})^{G'}$, but these are all linear, which completes the
proof.\end{proof}

In this paper, there are three main types of characters considered: PR-characters, Taketa-PR-characters and Taketa-characters. The following results shows that for a character of position $2$, these concepts coincide.

\begin{prop}\label{pos2}Let $\chi\in\Irr(G)$ with $\pos(\chi) = 2$. Then the following are equivalent:
\begin{enumerate}[(1)]
\item $\chi$ is a PR-character
\item $\chi$ is a Taketa-PR-character
\item $\chi$ is a Taketa-character
\end{enumerate}
\end{prop}

\begin{proof}$ $

$(1)\implies (2)$: Assume that $\chi$ is a PR-character and let $(G,H,\chi,\varphi)$ be a PRT. We then have $\pos(\varphi)+\posind{G}{H} \leq \pos(\chi) = 2$ so the only possibility is that $\pos(\varphi) = \posind{G}{H} = 1$ which means that $\varphi$ is linear and thus a Taketa character, so $\chi$ is a Taketa-PR-character.

$(2)\implies (3)$: Assume that $\chi$ is a Taketa-PR-character and let $(G,H,\chi,\varphi)$ be a PRT with $\varphi$ a Taketa-character. As above, we get $\pos(\varphi) = \posind{G}{H} = 1$ so by Proposition \ref{properties} (3) we must have $G'\leq H$ and hence $H\unlhd G$ so since $\pos(\varphi) = 1$ this holds for all constituents of $\chi_H$ and thus also for all constituents of $\chi_{G'}$. Hence restricting $\chi$ to $G'$ gives only linear constituents, and since the linear characters of $G'$ have $G''$ in their kernel, this shows that $G''\leq\rm{ker}(\chi)$, so $\chi$ is a Taketa-character.

$(3)\implies (1)$: Assume that $\chi$ is a Taketa-character. Since $\pos(\chi) = 2$ this means that $G''\leq\rm{ker}(\chi)$ so by Proposition \ref{taketa2} we have that $(G,G',\chi)$ is a PRT, and hence $\chi$ is a PR-character as claimed.
\end{proof}

\section{Position reducible groups}

In this section we will turn our attention to properties of the group $G$ related to the previously defined properties of the characters of $G$.

Since we have defined a certain type of character for $G$, it is natural to look at groups where all characters are of this type.

\begin{defi}[Position reducible group (PR-group)]$G$ is said to be a position
reducible group (PR-group) if all $\chi\in\Irr(G|G')$ are
PR-characters.\end{defi}

Note that we only require the non-linear irreducible characters to be PR-characters. This is because a linear character can never be a PR-character.

This also means that any abelian group will automatically be a PR-group, as it has no non-linear irreducible characters.

We now have three definitions with names containing PR (PRT, PR-character and PR-group). The PR part means essentially the same thing in all of these, and the concepts are very much related. To summarize, a PR-group is one where all non-linear irreducible characters are PR-characters, and PR-characters are those irreducible characters can can be part of a PRT. The two different ``versions'' of PRT that we have (one with four entries and one with 3 entries) are the same thing, where in one of them, we suppress part of the information.

\begin{cor}If $G$ is an M-group then $G$ is a PR-group.\end{cor}

\begin{proof}This is clear from Proposition \ref{monomialcharacter}.\end{proof}

One problem with PR-groups is that subgroups of PR-groups need not be PR-groups themselves. The same is true for M-groups, but where one can often say nice things about M-groups because the characters one considers for the subgroups are linear, this is not the case for PR-groups.

One could remedy this for PR-groups by requiring the character $\varphi$ in a PRT $(G,H,\chi,\varphi)$ to be linear. But this would be unnecessarily restrictive, as we can still say interesting things without this. Instead, we will consider the following groups.

\begin{defi}[Inductively position reducible group (IPR-group)]\label{ipr}$G$ is said to be
an inductively position reducible group (IPR-group) if for each
$\chi\in\Irr(G|G')$ there exists an $H < G$ such that $H$ is an
IPR-group and $(G,H,\chi)$ is a PRT.\end{defi}

Note that since we require $H < G$, the recursive nature of the
definition is not a problem (and just as for PR-groups, we can see that any abelian group is vacuously an IPR-group).

We see that if $P$ is some property such that any group satisfying $P$ is a PR-group and such that $P$ is inherited by subgroups, then any group satisfying $P$ is an IPR-group. In particular, by Proposition \ref{monomialcharacter} we see that if $P$ is a property such that any group satisfying $P$ is an M-group and such that $P$ is inherited by subgroups, then any group satisfying $P$ is an IPR-group. A special case of this is the following:

\begin{prop}\label{supersolvable}If $G$ is supersolvable, then $G$ is an IPR-group. In particular, if $G$ is nilpotent, then $G$ is an IPR-group.\end{prop}

\begin{proof}That a supersolvable group is an M-group follows from \cite[Theorem 6.22]{isaacs}, so the claim follows from the comments preceding the proposition, since any subgroup of a supersolvable group is itself supersolvable. \end{proof}

Emulating the definition of a PR-group, now using Taketa-PR-characters instead, we get the following.

\begin{defi}[Weak IPR-group]$G$ is said to be a weak IPR-group if all $\chi\in\Irr(G|G')$ are Taketa-PR-characters.\end{defi}

The justification for the term weak IPR-groups is given in Corollary \ref{weakipr}.

By the proof of Proposition \ref{monomialcharacter}, we see that any M-group is also a weak IPR-group (since all linear characters are Taketa-characters).

The following result is one of the main motivations for studying IPR-groups.

\begin{thm}\label{iprtaketa}If $G$ is an IPR-group then $G$ is a Taketa-group.\end{thm}

\begin{proof}Let $\chi\in\Irr(G)$ with $\pos(\chi) = i$. We then
need to show that $G^{(i)}\leq\ker(\chi)$.

The proof will proceed by induction on $i$ and $|G|$ (in the
lexicographic ordering of $\N\times\N$). If $|G| = 1$ the result is
trivial and if $i = 1$, $\chi$ is linear, so the result also holds
in this case. Assume therefore that $|G| > 1$ and $i > 1$.

Let $H < G$ be given such that $H$ is an IPR-group and let $\varphi\in\Irr(H)$ such that $(G,H,\chi,\varphi)$
is a PRT. This means that there is some $\psi\in\Irr(H)$ such that
$\posmax(\psi^G) + \pos(\varphi)\leq i$, so let such a $\psi$ be
given. In particular, we then have $\posmax(\psi^G) < i$.

We now get $\ker(\psi^G)\leq H$ by Lemma \ref{lm1}, but since this kernel is the
intersection of the kernels of its irreducible constituents and
since we have $\posmax(\psi^G) < i$, we get by induction that
$G^{(k)}\leq H$ where $k = \posmax(\psi^G)$ (since it must be
contained in the kernel of each irreducible constituent of
$\psi^G$).

Let $n = \pos(\varphi)$. Since
$|H| < |G|$ and $H$ is an IPR-group, we get by induction that
$H^{(n)}\leq\ker(\varphi)$.

Since we now have that $G^{(k)}\leq H$ and
$H^{(n)}\leq\ker(\varphi)$ we get that $G^{(n+k)}\leq\ker(\varphi)$.
Since $G^{(n+k)}$ is normal in $G$, we thus get that
$G^{(n+k)}\leq\ker(\varphi^G)$ by Lemma \ref{lm2}. Since $\varphi$ is a constituent of
$\chi_H$, we also have that $\chi$ is a constituent of $\varphi^G$ by Frobenius reciprocity,
so $\ker(\varphi^G)\leq\ker(\chi)$. Thus, $G^{(n+k)}\leq\ker(\chi)$,
and since we have $n + k = \pos(\varphi) + \posmax(\psi^G) \leq i$ this completes the proof.
\end{proof}

The following corollary justifies the use of the term weak IPR-group.

\begin{cor}\label{weakipr}If $G$ is an IPR-group then $G$ is a weak IPR-group.\end{cor}

\begin{proof}This is clear from Theorem \ref{iprtaketa}.\end{proof}

\begin{thm}If $G$ is a weak IPR-group then $G$ is a Taketa-group.\end{thm}

\begin{proof}Let $\chi\in\Irr(G)$ with $\pos(\chi) = i$. We need to show that $G^{(i)}\leq \rm{ker}(\chi)$. We will proceed by induction on $i$.

If $i= 1$ then the statement is trivial, so assume $i > 1$.

Let $H < G$ and $\varphi\in\Irr(H)$ be given such that $\varphi$ is a Taketa-character and $(G,H,\chi,\varphi)$ is a PRT. Let $n = \pos(\varphi)$, so $H^{(n)}\leq\rm{ker}(\varphi)$ since $\varphi$ is a Taketa-character.

Let $\psi\in\Irr(H)$ such that $\pos(\varphi) + \posmax(\psi^G)\leq i$. We have $\rm{ker}(\psi^G)\leq H$ by Lemma \ref{lm1} and since $\posmax(\psi^G) < i$ we have by induction that $G^{(k)}\leq \rm{ker}(\psi^G)\leq H$ where $k = \posmax(\psi^G)$.

We now have that $G^{(n+k)}\leq\rm{ker}(\varphi)$ and since $G^{(n+k)}$ is normal in $G$, this means that $G^{(n+k)}\leq\rm{ker}(\varphi^G)$ by Lemma \ref{lm2}. Since $\varphi$ is a constituent of $\chi_H$ we have that $\chi$ is a constituent of $\varphi^G$ by Frobenius reciprocity, so we get that $G^{(n+k)}\leq\rm{ker}(\chi)$, and since $n+k = \pos(\varphi) + \posmax(\psi^G) \leq i$ this completes the proof.
\end{proof}

An immediate consequence of Lemma \ref{xgr1} is the following.

\begin{prop}If $\rm{dl}(G)\leq 2$ then $G$ is an IPR-group.\end{prop}

\begin{proof}Since $G'$ is abelian, it is an IPR-group, and for any $\chi\in\Irr(G|G')$ it is clear from Lemma \ref{xgr1} that $(G,G',\chi)$ is a PRT. \end{proof}

If $G$ is a PR-group then $G$ need not be solvable, but we do have the following result, showing that it is at least not perfect (so in particular not simple).

\begin{cor}If $G$ is a non-abelian PR-group then $G'$ is not perfect. In particular, $G$ is not perfect.\end{cor}

\begin{proof}Let $\chi\in\Irr(G)$ with $\pos(\chi) = 2$. By Proposition \ref{pos2} we know that $G''\leq\rm{ker}(\chi)$. But since $\pos(\chi) \neq 1$ we have that $G'\not\leq\rm{ker}(\chi)$ so we cannot have $G'' = G'$ which proves the claim.\end{proof}

We can also use the above to get a bound on the derived length of a solvable PR-group in terms of the number of character degrees. The bound we obtain is slightly better than what is known for arbitrary solvable groups as long as the number of character degrees is not too large (see the introduction). First, we define a specific normal series for $G$ and recollect some of the properties of this.

\begin{defi}\label{normalseries}Let $i\geq 0$ be an integer. Define $$\rm{D}_i(G) = \bigcap_{\chi\in\Irr(G),\, \pos(\chi)\leq i}\rm{ker}(\chi)$$
\end{defi}

If $n = |\cd(G)|$ then it is now clear that we have a normal series $$\{1\} = \rm{D}_n(G) \leq \rm{D}_{n-1}(G)\leq\cdots\leq \rm{D}_1(G) \leq \rm{D}_0(G) = G$$ where we use the convention that an empty intersection of subgroups of $G$ is $G$ itself. It is also clear that we have $\rm{D}_1(G) = G'$ and that all the $\rm{D}_i(G)$ are normal in $G$. Some further properties of the $\rm{D}_i(G)$ are listed in the following theorem:

\begin{thm}\label{di}Assume $G$ is solvable and let $n = |\cd(G)|$.
\begin{enumerate}[(1)]
\item For all $1\leq i\leq n$ we have $\rm{dl}(\rm{D}_{i-1}(G)/\rm{D}_i(G))\leq 3$
\item If $\rm{dl}(\rm{D}_{i-1}(G)/\rm{D}_i(G)) = 3$ for some $1\leq i\leq n$ then $\rm{dl}(\rm{D}_{i-1}(G)/\rm{D}_{i+1}(G)) \leq 4$
\item If $i\leq 3$ then $\rm{dl}(\rm{D}_{n-i}(G))\leq i$
\end{enumerate}
\end{thm}

\begin{proof}
\begin{enumerate}[(1)]
\item This follows from \cite[Theorem 16.5]{manzwolf}.
\item This is \cite[Theorem 16.8]{manzwolf} (the conditions there are the same as here because of (1)).
\item This is \cite[Lemma 4.1]{kildetoft12}.
\end{enumerate}
\end{proof}

We also have the following corollary to Proposition \ref{pos2}.

\begin{cor}\label{d2}Assume that $G$ is a solvable PR-group. Then $\rm{dl}(G/\rm{D}_2(G))\leq 2$.
\end{cor}

\begin{proof}This follows directly from Proposition \ref{pos2}.
\end{proof}

Note that we just need to assume that all irreducible characters of position $2$ are PR-characters for the above proof to work.

We can then combine these statements to get the following. The idea of the proof is the same as in the proof of \cite[Theorem 1]{kildetoft12}.

\begin{thm}If $G$ is a solvable PR-group with $|\cd(G)| \geq 6$ then $\rm{dl}(G)\leq 2|\cd(G)| - 4$.
\end{thm}

\begin{proof}Let $n = |\cd(G)|$. Let $E_0 = \rm{D}_0(G) = G$, $E_1 =
\rm{D}_1(G) = G'$, $E_2 = \rm{D}_2(G)$ and define $E_i$ for $3\leq i\leq n-5$ by:

If $E_i$ has already been defined by a previous step, skip to $i+1$.

If $\rm{D}_{i-1}(G)''\leq \rm{D}_i(G)$, set $E_i = \rm{D}_i(G)$.

Otherwise, set $E_i = \rm{D}_{i-1}(G)''$ and $E_{i+1} = \rm{D}_{i+1}(G)$.

Set $E_{n-4} = \rm{D}_{n-4}(G)$ (this is consistent with the previous
rule), and $E_{n-3} = \rm{D}_{n-3}(G)$.

We now have a new normal series for $G$ given by $$\{1\}\leq
E_{n-3}\leq E_{n-4} \leq \cdots \leq E_1\leq E_0 = G$$

By Corollary \ref{d2} we have $\rm{dl}(E_0/E_2)\leq 2$, and by construction along with Theorem \ref{di} part 1, we get that
$\dl(E_{i-1}/E_i)\leq 2$ for $3\leq i \leq n-4$. By Theorem \ref{di} part 2, we also have $\dl(E_{n-4}/E_{n-3})\leq 3$. Finally, Theorem \ref{di} part 3 gives $\dl(E_{n-3})\leq 3$.

Combining this, we get $$\dl(G)\leq \dl(E_0/E_2) +
\sum_{i=3}^{n-4}\dl(E_{i-1}/E_i) + \dl(E_{n-4}/E_{n-3}) +
\dl(E_{n-3})$$ $$\leq 2 + (n-6)\cdot 2 + 3 + 3 = 2n-3 =
2|\cd(G)|-4$$ as was the claim.

\end{proof}

Note that the above proof also shows that if $G$ is a solvable PR-group with $5$ character degrees, then $\rm{dl}(G)\leq |\cd(G)|$. This is already known to hold for arbitrary solvable groups with $5$ character degrees however, as mentioned in the introduction.

\section{Some questions}

As previously noted, if $G$ is an M-group then $G$ is a weak IPR-group. But looking at the groups up to order $384$ in GAP, one can see that at least up to this order, all M-groups are in fact IPR-groups.

This is of course not a very large order to compute up to, but unfortunately due to the recursive nature of the definition of an IPR-group, it takes a lot of time and memory to do this for larger groups.

The reason it is hard to tell whether an M-group is necessarily an IPR-group is that not much is known about what conditions guarantee a subgroup of an M-group to be an M-group itself.

One thing to note is that the example given by Dade of an M-group which contains a normal subgroup which is not itself an M-group (\cite{dade73}), is in fact an IPR-group, but so is the normal subgroup which is not an M-group. This can be seen as follows: Let $G$ be Dade's example. Then $\cd(G) = \{1,2,7,14,16\}$, $\cd(G') = \{1,7,8\}$ and $\cd(G'') = \{1,8\}$. But $|G''| = 128$ so this subgroup is nilpotent, and the claim now follows from applying Lemma \ref{xgr2} twice. The same type of argument also gives the claim for the normal subgroup which is not an M-group. \\

Also worth noting is that all Taketa-groups of order at most $1000$, except possible those of order $768$, are PR-groups (by GAP computations), and as noted in Proposition \ref{pos2} there is at least some connection between the properties.

If it is indeed the case that any Taketa-group is a PR-group, then if $P$ is a property such that any group satisfying $P$ is a Taketa-group and such that $P$ is inherited by subgroups, then any group satisfying $P$ will be an IPR-group. An interesting special case of this would be that any group of odd order is an IPR-group, since these are Taketa-groups by the proof of \cite[Corollary 16.7]{manzwolf}.

\bibliography{bibtex-full}
\bibliographystyle{alpha}

\end{document}